\documentclass[final,a4paper]{siamltex}

\usepackage{amsmath,amsfonts,amssymb,mathtools}
\usepackage{graphicx,subfigure}
\usepackage{braket,dsfont,url,enumerate}

\usepackage{pdfpages} 
\usepackage{cmap}
\usepackage{hyperref}
\usepackage{color,dsfont,braket}
\usepackage{bbm}
\usepackage[headings]{fullpage}
\usepackage{fullpage}
\usepackage{fancyhdr}
\usepackage{times}
\usepackage{cmap}
\usepackage{pgf}
\usepackage{comment}

\newtheorem{remark}[theorem]{Remark}

\newcommand{\C}{\mathbb{C}}

\newcommand{\na}{\nabla}
\newcommand{\pa}{\partial}
\newcommand{\eps}{\varepsilon}
\newcommand{\om}{\omega}
\newcommand{\Om}{\Omega}
\newcommand{\De}{\Delta}
\newcommand{\si}{\sigma}

\newcommand{\IOm}{I\times \Om}
\newcommand{\ImOm}{I_m\times \Om}

\newcommand{\Id}{\operatorname{Id}}

\newcommand{\vertiii}[1]{{\left\vert\kern-0.25ex\left\vert\kern-0.25ex\left\vert #1
    \right\vert\kern-0.25ex\right\vert\kern-0.25ex\right\vert}}

\newcommand{\norm}[1]{\lVert#1\rVert}

\newcommand{\abs}[1]{\lvert#1\rvert}
\newcommand{\Ppol}[1]{\mathcal{P}_{#1}}

\newcommand{\half}{\frac{1}{2}}
\newcommand{\thalf}{\frac{3}{2}}

\newcommand{\R}{\mathbb{R}}
\newcommand{\lh}{\abs{\ln{h}}}
\newcommand{\lk}{\ln{\frac{T}{k}}}

\newcommand{\Xkh}{X^{q,r}_{k,h}}

\definecolor{darkred}{rgb}{.7,0,0}

\definecolor{green}{rgb}{0,0.7,0}

\begin{document}

\title{Pointwise best approximation results for Galerkin finite element solutions of parabolic problems}

\author{
Dmitriy Leykekhman\footnotemark[2]
\and
Boris Vexler\footnotemark[3]
}

\pagestyle{myheadings}
\markboth{DMITRIY LEYKEKHMAN AND BORIS VEXLER}{Parabolic pointwise best approximation}

\maketitle

\renewcommand{\thefootnote}{\fnsymbol{footnote}}
\footnotetext[2]{Department of Mathematics,
               University of Connecticut,
              Storrs,
              CT~06269, USA (dmitriy.leykekhman@uconn.edu). The author was partially supported by NSF grant DMS-1522555. }

\footnotetext[3]{Lehrstuhl f\"ur Mathematische Optimierung, Technische Universit\"at M\"unchen,
Fakult\"at f\"ur Mathematik, Boltzmannstra{\ss}e 3, 85748 Garching b. M\"unchen, Germany
(vexler@ma.tum.de). }

\renewcommand{\thefootnote}{\arabic{footnote}}


\begin{abstract}
In this paper we establish a best approximation property of fully discrete  Galerkin finite element solutions of second order parabolic problems on convex polygonal and polyhedral domains in the $L^\infty$ norm. The discretization method uses of continuous Lagrange finite elements in space and discontinuous Galerkin methods in time of an arbitrary order. The method of  proof differs from the established fully discrete error estimate techniques and for the first time allows to obtain such results in three space dimensions. It uses elliptic results, discrete resolvent estimates in weighted norms, and the discrete maximal parabolic regularity for discontinuous Galerkin methods established by the authors in \cite{LeykekhmanD_VexlerB_2015b}. In addition, the proof does not require any relationship between spatial mesh sizes  and time steps. We also establish an interior best approximation property that shows a more local behavior of the error at a given point.
\end{abstract}


\section{Introduction}
Let $\Om$ be a convex polygonal/polyhedral domains in $\mathbb{R}^N$, $N=2,3$ and $I=(0,T)$. We consider the second order parabolic problem
\begin{equation}\label{eq: heat equation}
\begin{aligned}
\pa_t u(t,x)-\Delta u(t,x) &= f(t,x), & (t,x) &\in \IOm,\;  \\
    u(t,x) &= 0,    & (t,x) &\in I\times\pa\Omega, \\
   u(0,x) &= u_0(x),    & x &\in \Omega.
\end{aligned}
\end{equation}
For the purpose of this paper we assume that $f$ and $u_0$ are such that the unique solution $u$ of~\eqref{eq: heat equation} fulfills $u\in C(\bar I\times \bar \Omega)\cap C(\bar I; H^1_0(\Omega))$. To achieve this, we can for example assume that the right-hand side $f \in L^r(I\times \Omega)$ with $r>\frac{N}{2}+1$ and $u_0\in C(\bar \Om)\cap H^1_0(\Omega)$, cf., e.\,g.,~\cite[Lemma 7.12]{Troeltzsch:2010}, but other assumptions are possible.

To discretize the problem we use  continuous Lagrange finite elements in space and discontinuous Galerkin methods in time. The precise description of the method is given in Section \ref{sec: discretization}.
Our main goal in this  paper is to establish global and interior space-time pointwise best approximation type results for the fully discrete error, namely,
\begin{equation}\label{eq: best approximation property}
\|u-u_{kh}\|_{L^\infty(I\times \Om)}\le C\lh\lk\|u-\chi\|_{L^\infty(I\times \Om)},
\end{equation}
where $u_{kh}$ denotes the fully discrete solution and $\chi$ is an arbitrary element of the finite dimensional space, $h$ is the  spatial mesh parameter  and $k$ stands for the maximal time step.  Such results have only natural assumptions on the problem data and
are desirable in many applications, for example in optimal control problems governed by parabolic equations.

Most of the work on pointwise error estimates for parabolic problems were devoted to establishing optimal convergence rates for the error between the exact solution $u(t)$ and the semidiscrete solution $u_h(t)$ that is
continuous in time,
\cite{BrambleJH_SchatzAH_ThomeeV_WahlbinLB_1977,ChenH_1993, DobrowolskiM_1980b, DobrowolskiM_1980a, LiB_2015a, LiB_SunW_2015a, NitscheJA_1979, NitscheJA_WheelerMF_1981, SammonP_1982, AHSchatz_VThomee_LBWahlbin_1998a, AHSchatz_VThomee_LBWahlbin_1980a, ThomeeV_WahlbinLB_2000a}. The best approximation results for the semidiscrete error
$u(t)-u_{h}(t)$ in $L^\infty(I\times \Om)$ norm can be found, for example, in \cite{LeykekhmanD_2004b, AHSchatz_VThomee_LBWahlbin_1998a}.

Results on fully discrete pointwise error estimates are much less abundant. Currently, there are several techniques available for obtaining fully discrete error estimates.
One popular technique splits the fully discrete error into two parts as  $u-u_{kh}=(u-u_{h})+(u_h-u_{kh})$. The first part of the error is estimated by the semidiscrete error estimates and the second part of the error is treated by using results from rational approximation of  analytic semigroups in Banach spaces.
Thus, for example, optimal convergence rates for backward Euler and Crank-Nicolson methods were obtained in \cite{AHSchatz_VThomee_LBWahlbin_1980a} (see also \cite[Sec. 9]{ThomeeV_2006} for treatment of general Pad\'{e} schemes).
A similar technique uses a different  splitting, $u-u_{kh}=(u-R_hu)+(R_hu-u_{kh})$, where $R_h$ is the Ritz projection. In this approach the first part of the error is treated by elliptic results and the second part of the error satisfies a certain parabolic equation with the right-hand side involving $(u-R_hu)$, which again can be treated by results from rational approximation of analytic semigroups in Banach spaces \cite{LeykekhmanD_WahlbinLB_2008a} (see also \cite[Thm. 8.6]{ThomeeV_2006}).
For smooth solutions, both  approaches above produce error estimates with optimal convergence rates. However, in many applications these two techniques require unreasonable assumptions on the data, as well as on the regularity of the  solution. As a result, the best approximation property~\eqref{eq: best approximation property} can not be derived,  except for the one-dimensional case \cite{WahlbinLB_1981a}.

Another approach, that is more direct, is based on the weighted technique. For $N=2$ and low order time schemes, this technique works rather well and allows one to obtain sharp results. Thus,  in \cite{ErikssonK_JohnsonC_1995a} (see also \cite[Thm. 4.1]{NochettoRH_VerdiC_1997a}) optimal convergence error estimates of the form
$$
\|u(t_n)-u_{kh}(t_n)\|_{L^\infty(\Om)}\le C\lh\left(\ln\frac{t_n}{k}\right)^{\frac{1}{2}}\max_{1\le m\le n}\left(k^q\|\pa^q_t u\|_{L^\infty((0,t_m)\times\Om)}+h^2\|D^2 u\|_{L^\infty((0,t_m)\times\Om)}\right),
$$
for piecewise constant and piecewise linear time discretizations, i.e. $q=1$ and $q=2$, correspondingly, were derived on convex polygonal domains (the result in  \cite{ErikssonK_JohnsonC_1995a} actually holds even on mildly graded meshes). The best approximation property of the form \eqref{eq: best approximation property} was derived in \cite{RRannacher_1990a} on convex polygonal domains  without any unnatural smoothness requirements. However, for $N=3$, the weighted technique is much more cumbersome and as of today, there is no three dimensional pointwise best approximation results or optimal error estimates even for backward Euler method.

In this paper  for the time discretization we consider discontinuous Galerkin (dG) methods of an arbitrary order. These methods were introduced to parabolic problems  in \cite{JametP_1978} and deeply analyzed in \cite{ErikssonK_JohnsonC_ThomeeV_1985}.
There are a number of important properties that make dG schemes attractive for temporal discretization of parabolic equations. For example,  such schemes allow for a priori error estimates of optimal order with respect to discretization parameters, such as the size of time steps, as well as with respect to the regularity requirements for the solution \cite{ErikssonK_JohnsonC_1991a,ErikssonK_JohnsonC_1995a}. Different systematic approaches for a posteriori error estimation and adaptivity developed for finite element discretizations can be adapted for dG temporal discretization of parabolic equations, see, e.\,g.,~\cite{SchmichM_VexlerB_2008, SchotzauD_WihlerTP_2010}. Since the trial space allows for discontinuities at the time nodes, the use of different spatial discretizations for each time step can be directly incorporated into the discrete formulation, see, e.\,g.,~\cite{SchmichM_VexlerB_2008}. Compared to the continuous Galerkin methods, dG schemes are not only A-stable but also strongly A-stable \cite{LasaintP_RaviartPA_1974}. An efficient and easy to implement approach that avoids complex coefficients, which arise in the equations obtained by a direct decoupling for high order dG schemes, was developed in~\cite{RichterT_SpringerA_VexlerB_2013}.

Our approach in establishing \eqref{eq: best approximation property} for dG methods is more in the spirit of the work of Palencia \cite{PalenciaC_1996a} and does not require semidiscrete error estimates or even any error splitting. Moreover, it does not require any relationship between the spatial mesh size $h$ and the maximal time step $k$, which is essential for problems on graded meshes.

Our approach is based on two main tools: The newly established discrete maximal parabolic regularity results~\cite{LeykekhmanD_VexlerB_2015b} for discontinuous Galerkin time schemes and discrete resolvent estimates of the following form:
\begin{equation}\label{eq: resolven estimate intro}
\|(z+\Delta_h)^{-1}\chi\|_{L^\infty(\Om)}\le  \frac{C}{|z|}\|\chi\|_{L^\infty(\Om)},\quad\text{for}\ z\in \mathbb{C}\setminus \Sigma_{\gamma},\quad \text{for all}\ \chi\in \mathbb{V}_h=V_h+iV_h,
\end{equation}
where $V_h$ is the space of continuous Lagrange finite elements and
\begin{equation}\label{eq: definition of sigma intro}
\Sigma_\gamma= \Set{z \in \mathbb{C} | \abs{\arg{(z)}} \le \gamma},
\end{equation}
for some $\gamma\in (0,\frac{\pi}{2})$ and the constant $C$ that may contain $\lh$ but must be independent of $h$ otherwise.
Such a discrete resolvent estimate can be shown directly \cite{BakaevNY_CrouzeixM_ThomeeV_2006a, BakaevNY_ThomeeV_WahlbinLB_2003a, LeykekhmanD_VexlerB_2015c} or by showing stability  and  smoothing results of the semidiscrete solution operator $E_h(t)=e^{-\Delta_ht}$ \cite{LiB_2015a, AHSchatz_VThomee_LBWahlbin_1998a}. The first approach is preferable since it establishes \eqref{eq: resolven estimate intro} for an arbitrary $\gamma\in (0,\frac{\pi}{2})$, while the second approach via theorem of Hille (see, e.g., Pazy \cite{PazyA_1983}, Thm. 2.5.2) only guarantees existence of some  $\gamma\in (0,\frac{\pi}{2})$.

In this paper we also establish a local version of the best approximation result \eqref{eq: best approximation property}. This  result (cf. Theorem \ref{thm:local_best_approx}) shows more local behavior of the error at a fixed point. For elliptic problems such estimates are well known  (cf. \cite{AHSchatz_LBWahlbin_1977a, AHSchatz_LBWahlbin_1995a, WahlbinLB_1991a}), but for parabolic problems the only result we are aware of is in \cite{RRannacher_1990a}, which is stated for convex polygonal domains without a proof and \cite{LeykekhmanD_VexlerB_2013a, LeykekhmanD_VexlerB_2015a} that are global in time. To obtain this result, in addition to the stability of the Ritz projection  in $L^\infty(\Om)$ norm and the resolvent estimate \eqref{eq: resolven estimate intro}, we need the following weighted resolvent estimate
\begin{equation}\label{eq: weighted resolven estimate intro}
\|\sigma^{\frac{N}{2}}(z+\Delta_h)^{-1}\chi\|_{L^2(\Om)}\le  \frac{C\lh}{|z|}\|\sigma^{\frac{N}{2}}\chi\|_{L^2(\Om)},\quad\text{for}\ z\in \mathbb{C}\setminus \Sigma_{\gamma},\quad \text{for all}\ \chi\in \mathbb{V}_h,
\end{equation}
with $\sigma(x)=\sqrt{|x-x_0|^2+K^2h^2}$. This estimate is established in Theorem \ref{thm:weightedResolvent}. The estimate \eqref{eq: weighted resolven estimate intro} is somewhat stronger than the corresponding resolvent estimate in $L^\infty$ norm, meaning that  \eqref{eq: resolven estimate intro} follows rather easily from \eqref{eq: weighted resolven estimate intro} (modulo logarithmic term $\lh$), but not vice versa.

The rest of the paper is organized as follows. In the next section we describe the discretization method and state our main results. In Section \ref{sec:elliptic}, we review some essential elliptic results in weighted norms. Section \ref{sec: weighted resolvent} is devoted to establishing resolvent estimate in weighted norms. In Section \ref{sec: max and smoothing}, we review some results from discrete maximal parabolic regularity. Finally, in Sections \ref{sec: proofs global results} and \ref{sec: proofs local results}, we give proofs of global and interior best approximation properties of the fully discrete solution.


\section{Discretization and statement of main results}\label{sec: discretization}
To introduce the time discontinuous Galerkin discretization for the problem,
 we partition  the interval $(0,T]$ into subintervals $I_m = (t_{m-1}, t_m]$ of length $k_m = t_m-t_{m-1}$, where $0 = t_0 < t_1 <\cdots < t_{M-1} < t_M =T$. The maximal and minimal time steps are denoted by $k =\max_{m} k_m$ and $k_{\min}=\min_{m} k_m$, respectively.
We impose the following conditions on the time mesh (as in ~\cite{LeykekhmanD_VexlerB_2015b} or ~\cite{DMeidner_RRannacher_BVexler_2011a}):
\begin{enumerate}[(i)]
  \item There are constants $c,\beta>0$ independent of $k$ such that
    \[
      k_{\min}\ge ck^\beta.
    \]
  \item There is a constant $\kappa>0$ independent of $k$ such that for all $m=1,2,\dots,M-1$
    \[
    \kappa^{-1}\le\frac{k_m}{k_{m+1}}\le \kappa.
    \]
  \item It holds $k\le\frac{1}{4}T$.
\end{enumerate}
The semidiscrete space $X_k^q$ of piecewise polynomial functions in time is defined by
\[
X_k^q=\Set{u_{k}\in L^2(I;H^1_0(\Om)) | u_{k}|_{I_m}\in \Ppol{q}(H^1_0(\Om)), \ m=1,2,\dots,M},
\]
where $\Ppol{q}(V)$ is the space of polynomial functions of degree $q$ in time with values in a Banach space $V$.
We will employ the following notation for functions in $X_k^q$
\begin{equation}\label{def: time jumps}
u^+_m=\lim_{\eps\to 0^+}u(t_m+\eps), \quad u^-_m=\lim_{\eps\to 0^+}u(t_m-\eps), \quad [u]_m=u^+_m-u^-_m.
\end{equation}
Next we define the following bilinear form
\begin{equation}\label{eq: bilinear form B}
 B(u,\varphi)=\sum_{m=1}^M \langle \pa_t u,\varphi \rangle_{I_m \times \Omega} + (\na u,\na \varphi)_{\IOm}+\sum_{m=2}^M([u]_{m-1},\varphi_{m-1}^+)_\Om+(u_{0}^+,\varphi_{0}^+)_\Om,
\end{equation}
where $( \cdot,\cdot )_{\Omega}$ and $( \cdot,\cdot )_{I_m \times \Omega}$ are the usual $L^2$ space and space-time inner-products,
$\langle \cdot,\cdot \rangle_{I_m \times \Omega}$ is the duality product between $ L^2(I_m;H^{-1}(\Omega))$ and $ L^2(I_m;H^{1}_0(\Omega))$. We note, that the first sum vanishes for $u \in X^0_k$. The dG($q$) semidiscrete (in time) approximation $u_k\in X_k^q$ of \eqref{eq: heat equation} is defined as
\begin{equation}\label{eq: semidiscrete heat with RHS}
B(u_k,\varphi_k)=(f,\varphi_k)_{\IOm}+(u_0,\varphi_{k,0}^+)_\Om \quad \text{for all }\; \varphi_k\in X_k^q.
\end{equation}
Rearranging the terms in \eqref{eq: bilinear form B}, we obtain an equivalent (dual) expression of $B$:
\begin{equation}\label{eq:B_dual}
 B(u,\varphi)= - \sum_{m=1}^M \langle u,\pa_t \varphi \rangle_{I_m \times \Omega} + (\na u,\na \varphi)_{\IOm}-\sum_{m=1}^{M-1} (u_m^-,[\varphi]_m)_\Om + (u_M^-,\varphi_M^-)_\Om.
\end{equation}

 Next we define the fully discrete approximation. For $h \in (0, h_0]$; $h_0 > 0$, let $\mathcal{T}$  denote  a quasi-uniform triangulation of $\Om$  with mesh size $h$, i.e., $\mathcal{T} = \{\tau\}$ is a partition of $\Om$ into cells (triangles or tetrahedrons) $\tau$ of diameter $h_\tau$ such that for $h=\max_{\tau} h_\tau$,
$$
\operatorname{diam}(\tau)\le h \le C |\tau|^{\frac{1}{N}}, \quad \forall \tau\in \mathcal{T}.
$$
 Let $V_h$ be the set of all functions in $H^1_0(\Om)$ that are polynomials of degree $r\in \mathbb{N}$ on each $\tau$, i.e. $V_h$ is the usual space of conforming finite elements.
To obtain the fully discrete approximation we consider the space-time finite element space
\begin{equation} \label{def: space_time}
\Xkh=\Set{v_{kh} \in L^2(I;V_h)| v_{kh}|_{I_m}\in \Ppol{q}(V_h), \ m=1,2,\dots,M}, \quad q\geq 0,\quad r\geq 1.
\end{equation}
We define a fully discrete $cG($r$)dG($q$)$ solution $u_{kh} \in \Xkh$ by
\begin{equation}\label{eq:fully discrete heat with RHS}
B(u_{kh},\varphi_{kh})=(f,\varphi_{kh})_{\IOm}+(u_0,\varphi_{kh}^+)_\Om \quad \text{for all }\; \varphi_{kh}\in \Xkh.
\end{equation}

\subsection{Main results}

Now we state our main results.

\subsubsection{Global pointwise best approximation error estimates}

The first result shows best approximation property of $cG($r$)dG($q$)$ Galerkin solution in $L^\infty(I\times \Om)$ norm. For $N=2$ and $q=0$, $r=1$, the result can be found in \cite{RRannacher_1990a} for convex polygonal domains. A similar result showing optimal error estimate is established in \cite{ErikssonK_JohnsonC_1995a}, Thm.~1.2. We are not aware of any pointwise best approximation type results for $N=3$.

\begin{theorem}[Global best approximation]\label{thm:global_best_approx}
Let $u$ and $u_{kh}$ satisfy \eqref{eq: heat equation} and \eqref{eq:fully discrete heat with RHS}
respectively. Then, there exists a constant $C$ independent of $k$ and  $h$ such that
\[
\norm{u-u_{kh}}_{L^\infty(I\times \Om)} \le C \lk \lh \inf_{\chi \in \Xkh} \norm{u-\chi}_{L^\infty(I\times \Om)}.
\]
\end{theorem}
The proof of this theorem is given in Section~\ref{sec: proofs global results}.

\subsubsection{Interior pointwise best approximation error estimates}

For the error at the point $x_0$ we can obtain a sharper result,  that shows  more localized behavior of the error at a fixed point. For elliptic problems similar results were obtained in~\cite{AHSchatz_LBWahlbin_1977a, AHSchatz_LBWahlbin_1995a}. We denote by $B_d=B_d(x_0)$ the  ball of radius $d$ centered at $x_0$.
\begin{theorem}[Interior best approximation]\label{thm:local_best_approx}
Let $u$ and $u_{kh}$ satisfy \eqref{eq: heat equation} and \eqref{eq:fully discrete heat with RHS}, respectively and let  $d>4h$. Let $\tilde t \in I_m$ with some $m \in \{1,2,\dots, M\}$ and $\overline{B}_d\subset\subset\Om$, then there exists a constant $C$ independent of $h$, $k$, and $d$ such that
$$
\begin{aligned}
|(u-u_{kh})(\tilde t,x_0)|\le C \lk \lh&\inf_{\chi \in \Xkh}\Biggl\{\|u-\chi\|_{L^\infty((0,t_m)\times B_d(x_0))}\\
 &+d^{-\frac{N}{2}}\bigg(\|u-\chi\|_{L^\infty((0,t_m);L^2(\Om))}+h\|\na(u-\chi)\|_{L^\infty((0,t_m);L^2(\Om))}\bigg)\Biggr\}.
\end{aligned}
$$
\end{theorem}
The proof of this theorem is given in Section~\ref{sec: proofs local results}.

\section{Elliptic estimates in weighted norms}\label{sec:elliptic}

In this section we collect some estimates for the finite element discretization of elliptic problems in weighted norms on convex polyhedral domains mainly taken from~\cite{LeykekhmanD_VexlerB_2015c}. These results will be used in the following sections within the proofs of Theorem~\ref{thm:global_best_approx} and Theorem~\ref{thm:local_best_approx}.

Let $x_0 \in \Omega$ be a fixed (but arbitrary) point. Associated with this point we introduce a smoothed Delta function \cite[~Appendix]{AHSchatz_LBWahlbin_1995a}, which we will denote by $\tilde{\delta}=\tilde{\delta}_{x_0}$. This function is supported in one cell, which is denoted by $\tau_{x_0}$ and satisfies
\begin{equation}\label{eq: definition delta}
(\chi, \tilde{\delta})_{\tau_{x_0}}=\chi({x_0}), \quad \forall \chi\in \Ppol{r}(\tau_{x_0}).
\end{equation}
In addition we also have
\begin{equation}\label{delta1}
 \|\tilde{\delta}\|_{W^{s,p}(\Om)} \le C h^{-s-N(1-\frac{1}{p})}, \quad 1\le p \le \infty, \quad s=0,1.
\end{equation}
Thus in particular $\|\tilde{\delta}\|_{L^1(\Om)} \le C$, $\norm{\tilde{\delta}}_{L^2(\Om)} \le Ch^{-\frac{N}{2}}$, and  $\|\tilde{\delta}\|_{L^\infty(\Om)} \le Ch^{-N}$.
Next we introduce a weight function
\begin{equation}\label{eq: sigma weight}
\sigma(x) = \sqrt{|x-x_0|^2+K^2h^2},
\end{equation}
where $K>0$ is a sufficiently large constant.
One can easily check that $\sigma$ satisfies the following properties:
\begin{subequations}
\begin{align}
\norm{\sigma^{-\frac{N}{2}}}_{L^2(\Om)}&\le C\lh^{\frac{1}{2}}, \label{eq: property 1 of sigma}\\
|\na \sigma|&\le C, \label{eq: property 2 of sigma}\\
|\na^2 \sigma|&\le C| \sigma^{-1}| \label{eq: property 3 of sigma}\\
\max_{x\in\tau}{\sigma}&\le C\min_{x\in\tau}{\sigma}, \quad \forall \tau \label{eq: property 4 of sigma}.
\end{align}
\end{subequations}
For the finite element space $V_h$ we will utilize the $L^2$ projection $P_h \colon L^2(\Omega) \to V_h$ defined by
\begin{equation}\label{eq:l2_proj}
(P_hv,\chi)_{\Om} = (v,\chi)_{\Om}, \quad \forall \chi\in V_h,
\end{equation}
the Ritz projection $R_h \colon H^1_0(\Omega) \to V_h$ defined by
\begin{equation}\label{eq:Ritz_proj}
(\nabla R_hv,\nabla \chi)_{\Om} = (\nabla v,\nabla \chi)_{\Om}, \quad \forall \chi\in V_h,
\end{equation}
and the usual nodal interpolation $i_h \colon C_0(\Omega) \to V_h$. Moreover we introduce the discrete Laplace operator $\Delta_h \colon V_h \to V_h$ defined by
\begin{equation}\label{eq:discreteLaplace}
(-\Delta_h v_h,\chi)_{\Om} = (\nabla v_h,\nabla \chi)_{\Om}, \quad \forall \chi\in V_h.
\end{equation}

The following lemma is a superapproximation result in weighted norms.
\begin{lemma}[Lemma 2.3 in~\cite{LeykekhmanD_VexlerB_2015c}]\label{lemma:super_ih_ph}
Let $v_h \in V_h$. Then the following estimates hold for any $\alpha,\beta \in \R$ and $K$ large enough:
\begin{equation}\label{sigma_est_ih3_ih4}
\norm{\sigma^\alpha(\Id-i_h)(\sigma^\beta v_h)}_{L^2(\Om)} + h \norm{\sigma^\alpha\nabla(\Id-i_h)(\sigma^\beta v_h)}_{L^2(\Om)} \le c h \norm{\sigma^{\alpha+\beta -1} v_h}_{L^2(\Om)},
\end{equation}
\begin{equation}\label{sigma_est_ph3_ph4}
\norm{\sigma^\alpha(\Id-P_h)(\sigma^\beta v_h)}_{L^2(\Om)}  + h \norm{\sigma^\alpha\nabla(\Id-P_h)(\sigma^\beta v_h)}_{L^2(\Om)} \le c h \norm{\sigma^{\alpha+\beta -1} v_h}_{L^2(\Om)}.
\end{equation}
\end{lemma}

The next lemma describes a connection between the regularized Delta functional $\tilde \delta$ and the weight $\sigma$.
\begin{lemma}\label{lemma:sigma_delta}
There holds
\begin{equation}\label{eq:sigma_delta}
\norm{\sigma^{\frac{N}{2}} \tilde \delta}_{L^2(\Om)} + h\norm{\sigma^{\frac{N}{2}} \nabla \tilde \delta}_{L^2(\Om)} + \norm{\sigma^{\frac{N}{2}} P_h\tilde \delta}_{L^2(\Om)} \le C.
\end{equation}
\end{lemma}
The proof of the above lemma for $N=2$, for example, can be found in \cite{ErikssonK_JohnsonC_1995a} and for $N=3$ in~\cite{LeykekhmanD_VexlerB_2015c}, Lemma 2.4.

The next result shows that the Ritz projection is almost stable in $L^\infty$ norm.
\begin{lemma}\label{lemma: stablity of Ritz}
There exists a constant $C>0$ independent on $h$, such that for any $v \in L^\infty(\Om)\cap H^1_0(\Omega)$,
$$
\|R_hv\|_{L^\infty(\Om)}\le C\lh\|v\|_{L^\infty(\Om)}.
$$
\end{lemma}
 For smooth domains such result was established in \cite{AHSchatz_LBWahlbin_1982a}, for polygonal domains in \cite{AHSchatz_1980a},
and for convex polyhedral domains  in~\cite[Thm.~3.1]{LeykekhmanD_VexlerB_2015c}. In the case of smooth domains or for convex polygonal domains the logarithmic factor can be removed for higher than piecewise linear order elements, i.e.  $r\geq 2$. The question of log-free stability result for convex polyhedral domains is still open.

 Next lemma is rather peculiar and can be thought as weighted Gagliardo-Nirenberg interpolation inequality.  The proof  is in \cite{LeykekhmanD_VexlerB_2015c},~Lemma 2.5.
\begin{lemma}\label{lemma: from weighted L2 to weighted in H1}
Let $N=3$. There exists a constant $C$ independent of $K$ and $h$ such that for any $f\in H^1_0(\Om)$, any $\alpha,\beta \in \mathbb{R}$ with $\alpha \ge -\frac{1}{2}$ and any $1\le p\le \infty$, $\frac{1}{p}+\frac{1}{p'}=1$ holds:
$$
\|\sigma^{\alpha}f\|^2_{L^2(\Om)}\le C\|\sigma^{\alpha-\beta}f\|_{L^p(\Om)} \|\sigma^{\alpha + 1 +\beta}\na f\|_{L^{p'}(\Om)},
$$
provided $\|\sigma^{\alpha-\beta}f\|_{L^p(\Om)}$ and  $\|\sigma^{\alpha + 1 +\beta}\na f\|_{L^{p'}(\Om)}$ are bounded.
\end{lemma}


\section{Weighted resolvent estimates}\label{sec: weighted resolvent}
In this section we will prove weighted resolvent estimates in two and three space dimensions. We will require such estimates to derive smoothing type estimates in the weighted norms in Section~\ref{sec: max and smoothing}.
Since in this section (only) we will be dealing with complex valued function spaces, we need to  modify the definition of the $L^2$-inner product as
$$
(u,v)_\Om = \int_\Omega u(x) \bar v(x) \, dx,
$$
where $\bar v$ is the complex conjugate of $v$ and the finite element space as $\mathbb{V}_h = V_h + i V_h$.

In the continuous case for Lipschitz domains the following result was shown in \cite{ShenZW_1995a}: For any $\gamma\in (0,\frac{\pi}{2})$ there exists a constant $C$ independent of $z$ such that
\begin{equation}\label{eq: continuous resolvent}
\|(z+\Delta)^{-1}v\|_{L^p(\Om)}\le  \frac{C}{1+\abs{z}}\|v\|_{L^p(\Om)},\quad  z \in \C \setminus \Sigma_{\gamma}, \quad 1\le p\le \infty, \quad v \in L^p(\Omega),
\end{equation}
where $\Sigma_{\gamma}$ is defined by
\begin{equation}\label{eq: definition of sigma}
\Sigma_\gamma= \Set{z \in \mathbb{C} | \abs{\arg{z}} \le \gamma}.
\end{equation}
In the finite element setting, it is also known that  for any $\gamma\in (0,\frac{\pi}{2})$ there exists a constant $C$ independent of $h$ and $z$ such that
\begin{equation}\label{eq: resolven estimate}
\|(z+\Delta_h)^{-1}\chi\|_{L^\infty(\Om)}\le  \frac{C}{1+|z|}\|\chi\|_{L^\infty(\Om)},\quad\text{for}\ z\in \mathbb{C}\setminus \Sigma_{\gamma},\quad \text{for all}\ \chi\in \mathbb{V}_h.
\end{equation}
For smooth domains such result is established in \cite{BakaevNY_ThomeeV_WahlbinLB_2003a} and for convex polyhedral domains with a constant containing $\lh$ in \cite{LeykekhmanD_VexlerB_2015c}. In \cite{LiB_2015a} the above resolvent result is established for convex polyhedral domains for some $\gamma\in (0,\frac{\pi}{2})$, but with a constant $C$ independent of $h$.

Our goal in this section is to establish the following resolvent estimate in the weighted norm.

\begin{theorem}\label{thm:weightedResolvent}
For any $\gamma\in(0,\frac{\pi}{2})$, there exists a constant $C$ independent of $h$ and $z$ such that
\[
\norm{\sigma^{\frac{N}{2}}(z+\Delta_h)^{-1}\chi}_{L^2(\Om)} \le \frac{C \lh}{\abs{z}} \norm{\sigma^{\frac{N}{2}} \chi}_{L^2(\Om)}, \quad\text{for}\ z\in \mathbb{C}\setminus \Sigma_{\gamma},
\]
for all $\chi \in \mathbb{V}_h$, where $\Sigma_{\gamma}$ is defined in \eqref{eq: definition of sigma}.
\end{theorem}

\subsection{Proof of Theorem \ref{thm:weightedResolvent} for \boldmath{$N=2$}}

For an arbitrary $\chi \in \mathbb{V}_h$ we define
\[
u_h = (z+\Delta_h)^{-1} \chi,
\]
or equivalently
\begin{equation}\label{eq: weak form zu_h N=2}
z(u_h,\varphi)-(\na u_h,\na \varphi)=(\chi,\varphi), \quad \forall \varphi\in \mathbb{V}_h.
\end{equation}
In this section the norm $\|\cdot\|$ will stand for $\|\cdot\|_{L^2(\Om)}$.
To estimate $ \|\sigma u_h\|$ we consider the expression
\begin{equation}\label{eq: expression from Thomee book N=2}
\|\sigma \na u_h\|^2=( \na(\sigma^2u_h),\na u_h)-2(\sigma\na\sigma u_h, \na u_h).
\end{equation}
By taking $\varphi=-P_h(\sigma^2u_h)$ in \eqref{eq: weak form zu_h N=2} and adding it to \eqref{eq: expression from Thomee book N=2}, we obtain
\begin{equation}\label{eq: estimates for z u_h with F N=2}
-z\|\sigma u_h\|^2+\|\sigma \na u_h\|^2=F,
\end{equation}
where
$$
F=F_1+F_2+F_3:=-(\sigma^2u_h,\chi)+(\na(\sigma^2u_h-P_h(\sigma^2u_h)),\na u_h)-2(\sigma\na\sigma u_h, \na u_h).
$$
Since $\gamma\le \abs{\arg{z}}\le \pi$, this equation is of the form
$$
e^{i\alpha}a + b = f,\quad \text{with}\quad a,b > 0,\quad 0\le |\alpha|\le \pi-\gamma,
$$
by multiplying it by $e^{-\frac{i\alpha}{2}}$ and taking real parts, we have
$$
a + b \le \left(\cos\left(\frac{\alpha}{2}\right)\right)^{-1}|f| \le \left(\sin\left(\frac{\gamma}{2}\right)\right)^{-1}|f| = C_\gamma|f|.
$$
From \eqref{eq: estimates for z u_h with F N=2} we therefore conclude
$$
\abs{z}\|\sigma u_h\|^2+\|\sigma \na u_h\|^2\le C_\gamma|F|, \quad\text{for}\ z\in \mathbb{C}\setminus \Sigma_{\gamma}.
$$
Using the Cauchy-Schwarz inequality and the arithmetic-geometric mean inequality we obtain,
$$
|F_1| = \abs{(\sigma^2 u_h,\chi)} \le \|\sigma u_h\|\|\sigma\chi\|\le CC_{\gamma}\abs{z}^{-1}\|\sigma\chi\|^2+\frac{\abs{z}}{2C_\gamma}\|\sigma u_h\|^2.
$$
To estimate $F_2$ we use Lemma \ref{lemma:super_ih_ph}, the Cauchy-Schwarz and the arithmetic-geometric mean inequalities,
$$
|F_2|\le \|\sigma^{-1}\na(\sigma^2u_h-P_h(\sigma^2u_h))\|\|\sigma\na u_h\|\le \frac{1}{4C_\gamma}\|\sigma\na u_h\|^2 + CC_\gamma\|u_h\|^2.
$$
Finally, using the properties of $\sigma$, we obtain
$$
|F_3|\le C\| u_h\|\|\sigma\na u_h\|\le  \frac{1}{4C_\gamma}\|\sigma\na u_h\|^2 + CC_\gamma\|u_h\|^2.
$$
Combining estimates for $F_i's$ and kicking back, we obtain
\begin{equation}\label{eq: sigma3/2 Gh after first step N=2}
\abs{z}\|\sigma u_h\|^2+\|\sigma \na u_h\|^2\le C_\gamma^2\left(\abs{z}^{-1}\norm{\sigma \chi}^2+\|u_h\|^2\right).
\end{equation}
Thus, in order to establish the desired weighted resolvent estimate, we need to show
\begin{equation}\label{eq: desired estimate for uh}
\| u_h\|^2\le C\lh^2\abs{z}^{-1}\norm{\sigma \chi}^2.
\end{equation}
To accomplish that, testing~\eqref{eq: weak form zu_h N=2} with $\varphi= u_h$, we obtain similarly as above
$$
\abs{z}\|u_h\|^2+\|\na u_h\|^2\le C_{\gamma}|f|, \quad\text{for}\quad z\in \C \setminus \Sigma_{\gamma},
$$
where $f=(\chi, u_h)$. Using the discrete  Sobolev  inequality (see \cite[~Lemma 1.1]{AHSchatz_VThomee_LBWahlbin_1980a}),
$$
\|v_h\|_{L^\infty(\Om)}\le C|\ln{h}|^{\frac{1}{2}}\|\na v_h\|_{L^2(\Om)}, \quad \forall v_h\in V_h,
$$
and using the property of $\sigma$ \eqref{eq: property 1 of sigma}, we obtain
$$
\begin{aligned}
\abs{z}\|u_h\|^2+\|\na u_h\|^2&\le C_{\gamma}\|\si\chi\|_{L^2(\Om)}\|\si^{-1}u_h\|_{L^2(\Om)}\\
&\le C_{\gamma}\|\sigma\chi\|_{L^2(\Om)}\|\sigma^{-1}\|_{L^2(\Om)}\| u_h\|_{L^\infty(\Om)}\\
&\le C_{\gamma}\lh\|\sigma\chi\|_{L^2(\Om)}\|\na u_h\|_{L^2(\Om)}\\
&\le C^2_{\gamma}\lh^2\|\sigma\chi\|^2_{L^2(\Om)}+\frac{1}{2}\|\na u_h\|^2_{L^2(\Om)}.
\end{aligned}
$$
Kicking back $\frac{1}{2}\|\na u_h\|^2_{L^2(\Om)}$, we establish \eqref{eq: desired estimate for uh} and hence Theorem \ref{thm:weightedResolvent} in the case of $N=2$.

\subsection{Proof of Theorem~\ref{thm:weightedResolvent} for \boldmath{$N=3$}}

The three dimensional case is more involved and we require
 some auxiliary results. For a given point $x_0 \in \Omega$, we introduce
 the adjoint regularized Green's function $G=G^{x_0}(x,\bar{z})$ by
\[
G = G^{x_0}(x,\bar{z})=(\bar{z}+\Delta)^{-1}\tilde{\delta}
\]
and its discrete analog $G_h=G_h^{x_0}(x,\bar{z}) \in \mathbb{V}_h$ by
\[
G_h = G_h^{x_0}(x,\bar{z})=(\bar{z}+\Delta_h)^{-1}P_h\tilde{\delta},
\]
which we can write in the weak form as
\begin{equation}\label{eq: weak form zG_h}
z(\varphi, G_h)-(\na \varphi,\na G_h)=(\varphi,\tilde{\delta}), \quad \forall \varphi\in \mathbb{V}_h.
\end{equation}

From~\cite{LeykekhmanD_VexlerB_2015c} we have the following result.
\begin{lemma}[\cite{LeykekhmanD_VexlerB_2015c}]\label{lemma:G_h_L_3}
Let $G_h \in \mathbb{V}_h$ be defined by~\eqref{eq: weak form zG_h}. There holds
\[
\|G_h\|_{L^3(\Om)}\le C\lh^{\frac{1}{3}}.
\]
\end{lemma}

\begin{lemma}\label{lemma:resolvent_l32_linf}
Let $w_h \in \mathbb{V}_h$ be the solution of
\[
z(w_h,\varphi)-(\na w_h,\na \varphi)=(f,\varphi), \quad \forall \varphi\in \mathbb{V}_h
\]
for some $f \in L^{\thalf}(\Omega)$. There exists a constant $C>0$ such that
\[
\norm{w_h}_{L^\infty(\Omega)} \le C \lh^\frac{1}{3} \norm{f}_{L^\thalf(\Omega)}.
\]
\end{lemma}
\begin{proof}
There holds
$$
w_h(x_0) = z(w_h,G_h)-(\na w_h,\na G_h)=(f,G_h).
$$
Hence,
\[
\abs{w_h(x_0)} = \abs{(f,G_h)} \le \norm{f}_{L^\thalf(\Omega)} \, \|G_h\|_{L^3(\Om)}.
\]
Applying Lemma~\ref{lemma:G_h_L_3} we obtain the result.
\end{proof}

\begin{lemma}\label{lemma:resolvent_l3}
Let $v_h \in \mathbb{V}_h$ be the solution of
\[
z(v_h,\varphi)-(\na v_h,\na \varphi)=(f,\varphi), \quad \forall \varphi\in \mathbb{V}_h,
\]
and $f\in L^1(\Om)$.
There exists a constant $C>0$ such that
\[
\norm{v_h}_{L^3(\Omega)} \le C \lh^\frac{1}{3} \norm{f}_{L^1(\Omega)}.
\]
\end{lemma}
\begin{proof}
We consider a dual solution $w_h \in \mathbb{V}_h$ defined by
\[
z(\varphi, w_h)-(\na \varphi,\na w_h)=(\varphi,v_h|v_h|), \quad \forall \varphi\in \mathbb{V}_h.
\]
There holds
\[
\norm{v_h}_{L^3(\Omega)}^3 =z(v_h,w_h)-(\na v_h,\na w_h)= (f,w_h) \le \norm{f}_{L^1(\Omega)} \norm{w_h}_{L^\infty(\Omega)}.
\]
By Lemma~\ref{lemma:resolvent_l32_linf} that also holds for the adjoint problem,  we have
\[
\norm{w_h}_{L^\infty(\Omega)} \le C \lh^{\frac{1}{3}} \norm{v_h|v_h|}_{L^\thalf(\Omega)} \le C \lh^{\frac{1}{3}} \norm{v_h}^2_{L^3(\Omega)}.
\]
Thus, we get
\[
\norm{v_h}_{L^3(\Omega)}^3 \le C \lh^{\frac{1}{3}}  \norm{f}_{L^1(\Omega)} \norm{v_h}^2_{L^3(\Omega)}.
\]
Canceling $\norm{v_h}^2_{L^3(\Omega)}$ completes the proof.
\end{proof}

With these results we proceed with the proof of Theorem~\ref{thm:weightedResolvent}  for $N=3$.
\begin{proof}
For an arbitrary $\chi \in \mathbb{V}_h$ we define
\[
u_h = (z+\Delta_h)^{-1} \chi.
\]
or equivalently
\begin{equation}\label{eq: weak form zu_h}
z(u_h,\varphi)-(\na u_h,\na \varphi)=(\chi,\varphi), \quad \forall \varphi\in \mathbb{V}_h.
\end{equation}
To estimate $ \|\sigma^{\frac{3}{2}}u_h\|$ we consider the expression
\begin{equation}\label{eq: expression from Thomee book}
\|\sigma^{\frac{3}{2}}\na u_h\|^2=( \na(\sigma^3u_h),\na u_h)-3(\sigma^2\na\sigma u_h, \na u_h).
\end{equation}
By taking $\varphi=-P_h(\sigma^3u_h)$ in \eqref{eq: weak form zu_h} and adding to \eqref{eq: expression from Thomee book}, we obtain
\begin{equation}\label{eq: estimates for z u_h with F}
-z\|\sigma^{\frac{3}{2}}u_h\|^2+\|\sigma^{\frac{3}{2}}\na u_h\|^2=F,
\end{equation}
where
$$
F=F_1+F_2+F_3:=-(P_h(\sigma^3u_h),\chi)+(\na(\sigma^3u_h-P_h(\sigma^3u_h)),\na u_h)-3(\sigma^2\na\sigma u_h, \na u_h).
$$
Since $\gamma\le \abs{\arg{z}}\le \pi$, this equation is of the form
$$
e^{i\alpha}a + b = f,\quad \text{with}\quad a,b > 0,\quad 0\le |\alpha|\le \pi-\gamma,
$$
by multiplying it by $e^{-\frac{i\alpha}{2}}$ and taking real parts, we have
$$
a + b \le \left(\cos\left(\frac{\alpha}{2}\right)\right)^{-1}|f| \le \left(\sin\left(\frac{\gamma}{2}\right)\right)^{-1}|f| = C_\gamma|f|.
$$
From \eqref{eq: estimates for z u_h with F} we therefore conclude
$$
\abs{z}\|\sigma^{\frac{3}{2}}u_h\|^2+\|\sigma^{\frac{3}{2}}\na u_h\|^2\le C_\gamma|F|,\quad\text{for}\ z\in \mathbb{C}\setminus \Sigma_{\gamma}.
$$
Using the Cauchy-Schwarz inequality and the arithmetic-geometric mean inequality we obtain,
$$
|F_1| = \abs{(\sigma^3 u_h,\chi)} \le \|\sigma^{\frac{3}{2}}u_h\|\|\sigma^{\frac{3}{2}}\chi\|\le CC_{\gamma}\abs{z}^{-1}\|\sigma^{\frac{3}{2}}\chi\|^2+\frac{\abs{z}}{2C_\gamma}\|\sigma^{\frac{3}{2}}u_h\|^2.
$$
To estimate $F_2$ we use Lemma \ref{lemma:super_ih_ph}, the Cauchy-Schwarz and the arithmetic-geometric mean inequalities,
$$
|F_2|\le \|\sigma^{-\frac{3}{2}}\na(\sigma^3u_h-P_h(\sigma^3u_h))\|\|\sigma^{\frac{3}{2}}\na u_h\|\le \frac{1}{4C_\gamma}\|\sigma^{\frac{3}{2}}\na u_h\|^2 + CC_\gamma\|\sigma^{\frac{1}{2}}u_h\|^2.
$$
Finally, using the properties of $\sigma$, we obtain
$$
|F_3|\le C\|\sigma^{\frac{1}{2}} u_h\|\|\sigma^{\frac{3}{2}}\na u_h\|\le  \frac{1}{4C_\gamma}\|\sigma^{\frac{3}{2}}\na u_h\|^2 + CC_\gamma\|\sigma^{\frac{1}{2}}u_h\|^2.
$$
Combining the estimates for $F_i's$ and kicking back, we obtain
\begin{equation}\label{eq: sigma3/2 Gh after first step}
\abs{z}\|\sigma^{\frac{3}{2}}u_h\|^2+\|\sigma^{\frac{3}{2}}\na u_h\|^2\le C\left(\abs{z}^{-1}\norm{\sigma^\thalf \chi}^2+\|\sigma^{\frac{1}{2}}u_h\|^2\right).
\end{equation}
Thus, in order to establish the desired weighted resolvent estimate, we need to show
\begin{equation}\label{eq: desired estimate for sigma1/2 uh}
\|\sigma^{\frac{1}{2}}u_h\|^2\le C\lh^2\abs{z}^{-1}\norm{\sigma^\thalf \chi}^2.
\end{equation}
To accomplish that, we consider the expression
$$
-z\|\sigma^{\frac{1}{2}}u_h\|^2+\|\sigma^{\frac{1}{2}}\na u_h\|^2=-z(u_h, \sigma u_h)+(\na u_h, \na(\sigma u_h))-(\na u_h, \na\sigma u_h).
$$
Testing~\eqref{eq: weak form zu_h} with $\varphi=P_h(\sigma u_h)$ we obtain similarly as above
$$
\abs{z}\|\sigma^{\frac{1}{2}}u_h\|^2+\|\sigma^{\frac{1}{2}}\na u_h\|^2\le C_{\gamma}|f|, \quad\text{for}\quad z\in \C \setminus \Sigma_{\gamma},
$$
where
$$
f=f_1+f_2+f_3:=-(P_h(\sigma u_h),\chi)+(\na(\sigma u_h-P_h(\sigma u_h)),\na u_h)-(\na\sigma u_h,\na u_h).
$$
Using the Cauchy-Schwarz inequality and the arithmetic-geometric mean inequality, we obtain
$$
|f_1| = \abs{(\sigma u_h,\chi)} \le \|\sigma^{-\frac{1}{2}}u_h\|\|\sigma^{\frac{3}{2}}\chi\|\le \half \|\sigma^{-\frac{1}{2}}u_h\|^2 +  \half \norm{\sigma^\thalf \chi}^2.
$$
To estimate $f_2$ we use Lemma \ref{lemma:super_ih_ph}, the Cauchy-Schwarz and the arithmetic-geometric mean inequalities,
$$
|f_2|\le \|\sigma^{-\frac{1}{2}}\na(\sigma u_h-P_h(\sigma u_h))\|\|\sigma^{\frac{1}{2}}\na u_h\|\le \frac{1}{4C_\gamma}\|\sigma^{\frac{1}{2}}\na u_h\|^2 + CC_\gamma\|\sigma^{-\frac{1}{2}}u_h\|^2.
$$
Finally, using the properties of $\sigma$, we obtain
$$
|f_3|\le C\|\sigma^{-\frac{1}{2}} u_h\|\|\sigma^{\frac{1}{2}}\na u_h\|\le  \frac{1}{4C_\gamma}\|\sigma^{\frac{1}{2}}\na u_h\|^2 + CC_\gamma\|\sigma^{-\frac{1}{2}}u_h\|^2.
$$
Combining estimates for $f_i's$ and kicking back, we obtain
\begin{equation}\label{eq: sigma1/2 uh after second step}
\abs{z}\|\sigma^{\frac{1}{2}}u_h\|^2+\|\sigma^{\frac{1}{2}}\na u_h\|^2 \le C \left(\|\sigma^{-\frac{1}{2}}u_h\|^2 + \norm{\sigma^\thalf \chi}^2\right).
\end{equation}
To estimate $\|\si^{-\frac{1}{2}} u_h\|$ we use Lemma \ref{lemma: from weighted L2 to weighted in H1} with $\alpha=\beta=-\frac{1}{2}$ and $p=3$, to obtain
\begin{equation}\label{eq: from L2 using lemma resolvent}
\|\si^{-\frac{1}{2}} u_h\|\le C\|u_h\|^{\frac{1}{2}}_{L^3(\Om)}\|\na u_h\|^{\frac{1}{2}}_{L^{\frac{3}{2}}(\Om)}.
\end{equation}
Using Lemma~\ref{lemma:resolvent_l3}, we have
\[
\|u_h\|_{L^3(\Om)} \le C \lh^\frac{1}{3} \norm{\chi}_{L^1(\Omega)} \le C \lh^\frac{1}{3} \norm{\sigma^{-\thalf}}\norm{\sigma^\thalf \chi} \le C \lh^\frac{5}{6} \norm{\sigma^\thalf \chi}.
\]
To estimate $\|\na u_h\|_{L^{\frac{3}{2}}(\Om)}$ we proceed by the H\"older inequality
\begin{equation}\label{eq: estimates grad u_h in L3/2 resolvent}
\|\na u_h\|_{L^{\frac{3}{2}}(\Om)}\le C\lh^{\frac{1}{6}}\|\si^{\frac{1}{2}}\na u_h\|_{L^2(\Om)}.
\end{equation}
Thus, using \eqref{eq: sigma1/2 uh after second step} and the above estimates, we have
\[
\begin{aligned}
\abs{z}\|\sigma^{\frac{1}{2}}u_h\|^2+\|\sigma^{\frac{1}{2}}\na u_h\|^2&\le C\left(\| u_h\|_{L^3(\Om)}\|\na u_h\|_{L^{\frac{3}{2}}(\Om)}+\norm{\sigma^\thalf \chi}^2\right)\\
&\le C\left(\lh\|\sigma^{\frac{1}{2}}\na u_h\|\norm{\sigma^\thalf \chi}+\norm{\sigma^\thalf \chi}^2\right)\\
&\le C\lh^2\norm{\sigma^\thalf \chi}^2+\frac{1}{2}\|\sigma^{\frac{1}{2}}\na u_h\|^2.
\end{aligned}
\]
Kicking back $\|\si^{\frac{1}{2}}\na u_h\|^2$, we finally obtain
$$
 \|\si^{\frac{1}{2}} u_h\|^2\le C\lh^2\abs{z}^{-1} \norm{\sigma^\thalf \chi}^2,
$$
which shows \eqref{eq: desired estimate for sigma1/2 uh} and hence the theorem.
\end{proof}

\section{Maximal parabolic and smoothing estimates}\label{sec: max and smoothing}

In this section we state some smoothing and stability results for homogeneous and inhomogeneous problems that are central in establishing our main results.
Since we apply the following results for different norms on $V_h$, namely, for $L^p(\Om)$ and weighted $L^2(\Om)$ norms, we state them for a general norm $\vertiii{\cdot}$.

Let $\vertiii{\cdot}$ be a norm on $V_h$ (extended in a straightforward way to a norm on $\mathbb{V}_h$) such that for some $\gamma\in(0,\frac{\pi}{2})$ the following resolvent estimate holds,
\begin{equation}\label{eq: resolvent in Banach space}
\vertiii{(z+\Delta_h)^{-1}\chi} \le \frac{M_h}{\abs{z}} \vertiii{\chi},\quad\text{for}\ z\in \mathbb{C}\setminus \Sigma_{\gamma},
\end{equation}
for all $\chi \in \mathbb{V}_h$, where $\Sigma_{\gamma}$ is defined in \eqref{eq: definition of sigma} and
 the constant $M_h$ is independent of $z$.

This assumption is fulfilled for $\vertiii{\cdot}=\norm{\cdot}_{L^p(\Omega)}$, $1 \le p \le \infty$, with a constant $M_h \le C$ independent of $h$, see~\cite{LiB_SunW_2015a}, and for $\vertiii{\cdot}=\norm{\sigma^{\frac{N}{2}}\cdot}_{L^2(\Omega)}$ with $M_h \le C \lh$, see Theorem~\ref{thm:weightedResolvent}.

\subsection{Smoothing estimates for the homogeneous problem in Banach spaces}

First, we consider the homogeneous heat equation~\eqref{eq: heat equation}, i.e.  with $f=0$ and its discrete approximation $u_{kh} \in \Xkh$ defined by
\begin{equation}\label{eq: dg(r) homogeneous}
B(u_{kh},\varphi_{kh}) = (u_0,\varphi_{kh,0}^+) \quad \forall \varphi_{kh} \in \Xkh.
\end{equation}
The first result is a smoothing type estimate, see~\cite[Theorem~13]{LeykekhmanD_VexlerB_2015b}, cf. also~\cite[Thmeorem~5.1]{ErikssonK_JohnsonC_LarssonS_1998a} for the case of the $L^2$ norm.
\begin{lemma}[Fully discrete homogeneous smoothing estimate]\label{lemma: homogeneous smoothing dG_r fully discrete}
Let $\vertiii{\cdot}$ be a norm on $V_h$ fulfilling the resolvent estimate~\eqref{eq: resolvent in Banach space}. Let $u_{kh}$ be the solution  of \eqref{eq: dg(r) homogeneous}. Then, there exists a constant $C$ independent of $k$ and $h$ such that
$$
\sup_{t\in I_m}\vertiii{\pa_t u_{kh}(t)}+\sup_{t\in I_m}\vertiii{\Delta_h u_{kh}(t)}+k_m^{-1}\vertiii{[u_{kh}]_{m-1}}\le \frac{CM_h}{t_m}\vertiii{P_h u_0},
$$
for $m=1,2,\dots,M$. For $m=1$ the jump term is understood as $[u_{kh}]_0 = u_{kh,0}^+-P_h u_0$.
\end{lemma}


For the proofs of Theorem~\ref{thm:global_best_approx} and Theorem~\ref{thm:local_best_approx}, we will need an additional stability result, which is also formulated for a general norm $\vertiii{\cdot}$ fulfilling~\eqref{eq: resolvent in Banach space}.

\begin{lemma}\label{lemma: smoothing L1 in time}
Let $\vertiii{\cdot}$ be a norm on $V_h$ fulfilling the resolvent estimate~\eqref{eq: resolvent in Banach space}. Let $u_{kh}$ be the solution  of \eqref{eq: dg(r) homogeneous}.
Then there exists a constant $C$ independent of $k$ and $h$ such that
$$
\sum_{m=1}^M\left(\int_{I_m}\vertiii{\pa_t u_{kh}(t)}dt+\int_{I_m}\vertiii{\Delta_h u_{kh}(t)}dt+\vertiii{[u_{kh}]_{m-1}}\right)\le CM_h\lk\vertiii{P_h u_0}.
$$
For $m=1$ the jump term is understood as $[u_{kh}]_0 = u_{kh,0}^+-P_h u_0$.
\end{lemma}
\begin{proof}
Using the above smoothing result, we have
$$
\begin{aligned}
&\sum_{m=1}^M\left(\int_{I_m}\vertiii{\pa_t u_{kh}(t)}dt+\int_{I_m}\vertiii{\Delta_h u_{kh}(t)}dt+\vertiii{[u_{kh}]_{m-1}}\right)\\
&\le\sum_{m=1}^M k_m\left(\sup_{t\in I_m}\vertiii{\pa_t u_{kh}(t)}+\sup_{t\in I_m}\vertiii{\Delta_h u_{kh}(t)}+k_m^{-1}\vertiii{[u_{kh}]_{m-1}}\right)\\
& \le C M_h \sum_{m=1}^M \frac{k_m}{t_m}\vertiii{P_h u_{0}}\le C M_h \lk\vertiii{P_h u_{0}},
\end{aligned}
$$
where in the last step we used that $\sum_{m=1}^M \frac{k_m}{t_m}\le C\lk$.
\end{proof}


\subsection{Discrete maximal parabolic estimates for the inhomogeneous problem in Banach spaces}

Now, we consider the inhomogeneous heat equation \eqref{eq: heat equation},  with $u_0=0$ and its discrete approximation $u_{kh} \in \Xkh$ defined by
\begin{equation}\label{eq: dGr nonhomogeneous equation fully}
B(u_{kh},\varphi_{kh})=(f,\varphi_{kh}),\quad \forall \varphi_{kh}\in \Xkh.
\end{equation}
The following discrete maximal parabolic regularity result is taken from~\cite[Theorem 14]{LeykekhmanD_VexlerB_2015b}.
\begin{lemma}[Discrete maximal parabolic regularity]\label{lemma: fully discrete_maximal_parabolic}
Let $\vertiii{\cdot}$ be a norm on $V_h$ fulfilling the resolvent estimate~\eqref{eq: resolvent in Banach space} and let $1 \le s \le \infty$. Let $u_{kh}$ be a solution  of \eqref{eq: dGr nonhomogeneous equation fully}. Then, there exists a constant $C$ independent of $k$ and $h$ such that
\begin{equation*}
\begin{aligned}
\left(\sum_{m=1}^M\int_{I_m}\vertiii{\pa_t u_{kh}(t)}^sdt\right)^{\frac{1}{s}}+\left(\sum_{m=1}^M\int_{I_m}\vertiii{\Delta_h u_{kh}(t)}^sdt\right)^{\frac{1}{s}}&+\left(\sum_{m=1}^Mk_m\vertiii{k_m^{-1}[u_{kh}]_{m-1}}^s\right)^{\frac{1}{s}}\\
&\le C M_h \lk\left(\int_I\vertiii{P_h f(t)}^sdt\right)^{\frac{1}{s}},
\end{aligned}
\end{equation*}
with obvious notation change in the case of $s=\infty$. For $m=1$ the jump term is understood as $[u_{kh}]_0=u_{kh,0}^+$.
\end{lemma}


\begin{remark}\label{remark:M_h}
As mentioned above the assumption~\eqref{eq: resolvent in Banach space} is fulfilled for $\vertiii{\cdot} = \norm{\cdot}_{L^p(\Omega)}$ and any $1\le p\le \infty$ with $M_h \le C$ and for $\vertiii{\cdot}=\norm{\sigma^{\frac{N}{2}}\cdot}_{L^2(\Omega)}$ with $M_h \le C \lh$. Therefore the results of Lemma~\ref{lemma: homogeneous smoothing dG_r fully discrete}, Lemma~\ref{lemma: smoothing L1 in time}, and Lemma~\ref{lemma: fully discrete_maximal_parabolic} are fulfilled for these two choices of norms with the corresponding constants $M_h$.
\end{remark}
\section{Proof of Theorem \ref{thm:global_best_approx}}\label{sec: proofs global results}

Let $\tilde t \in(0,T]$ and let $x_0 \in \Omega$ be an arbitrary but fixed point. Without loss of generality we assume $\tilde t \in (t_{M-1},T]$. We consider two cases: $\tilde t = T$ and $t_{M-1} < \tilde t < T$.

\textbf{Case 1, \boldmath{$\tilde t = T$}:}
To establish our result we will estimate $u_{kh}(T,x_0)$
by using a duality argument. First, we define $g$ to be a solution to the following backward parabolic problem
\begin{equation}\label{eq: heat with dirac initial}
\begin{aligned}
-\pa_t g(t,x)-\Delta g(t,x) &= 0& (t,x)\in \IOm,\;  \\
g(t,x)&=0, &(t,x) \in I\times\pa\Omega, \\
g(T,x)&=\tilde{\delta}_{x_0}, & x \in \Omega,
\end{aligned}
\end{equation}
where $\tilde \delta = \tilde{\delta}_{x_0}$ is the smoothed Dirac function introduced in \eqref{eq: definition delta}.
Let $g_{kh} \in \Xkh$ be the corresponding cG($r$)dG($q$) solution defined by
\begin{equation}\label{eq: discrete heat with dirac initial}
B(\varphi_{kh}, g_{kh})=\varphi_{kh}(T,x_0)\quad \forall \varphi_{kh}\in \Xkh.
\end{equation}
Then using that cG($r$)dG($q$) method is consistent, we have
\begin{equation}\label{eq: starting with B pointwise time}
\begin{aligned}
u_{kh}(T,x_0) &= B(u_{kh}, g_{kh})=B(u, g_{kh})\\
&=-\sum_{m=1}^{M}(u,\pa_t g_{kh})_{\ImOm}+(\na  u, \na g_{kh})_{\IOm}-\sum_{m=1}^{M-1}(u_m,[g_{kh}]_m)_{\Om}+(u(T),g_{kh,M}^-)_\Om\\
&=J_1+J_2+J_3+J_4.
 \end{aligned}
\end{equation}
Using the H\"{o}lder inequality we have
\begin{equation}\label{eq: estimate for J1 pointwise time}
\begin{aligned}
    J_1&\le \sum_{m=1}^{M}\|u\|_{L^\infty(I_m \times \Om)}\|\pa_t g_{kh}\|_{L^1(I_m; L^1(\Om))}\\
&\le \|u\|_{L^\infty(I \times \Om)}\sum_{m=1}^{M}\|\pa_t g_{kh}\|_{L^1(I_m; L^1(\Om))}.
\end{aligned}
\end{equation}
For $J_2$ we obtain using the stability of the Ritz projection in $L^\infty(\Om)$ norm on polygonal and polyhedral domains, see Lemma~\ref{lemma: stablity of Ritz},
\begin{equation}\label{eq: estimate for J2 pointwise time}
\begin{aligned}
    J_2&= (\na  R_hu, \na g_{kh})_{\IOm}=-(R_hu, \Delta_h g_{kh})_{\IOm}\\
&\le \|R_hu\|_{L^\infty(I \times \Om)}\|\Delta_h g_{kh}\|_{L^{1}(I; L^1(\Om))}\\
& \le C \lh \|u\|_{L^\infty(I \times \Om)}\|\Delta_h g_{kh}\|_{L^{1}(I; L^1(\Om))}
\end{aligned}
\end{equation}
For $J_3$ and $J_4$ we obtain
\begin{equation}\label{eq: estimate for J3 pointwise time}
\begin{aligned}
    J_3&\le \sum_{m=1}^{M-1}\|u_m\|_{L^\infty(\Om)}\|[g_{kh}]_m\|_{L^1(\Om)}
 \le \|u\|_{L^\infty(I \times \Om)} \sum_{m=1}^{M-1}\|[g_{kh}]_m\|_{L^1(\Om)},\\
J_4 &\le\|u(T)\|_{L^\infty(\Om)}\|g_{kh,M}^{-}\|_{L^1(\Om)}
\le \|u\|_{L^\infty(I \times \Om)} \|g_{kh,M}^{-}\|_{L^1(\Om)}.
\end{aligned}
\end{equation}
Combining the estimates for $J_1$, $J_2$, $J_3$, and $J_4$ and applying Lemma \ref{lemma: smoothing L1 in time} with $\vertiii{\cdot}=\|\cdot\|_{L^1(\Om)}$ and $M_h \le C$, cf. Remark~\ref{remark:M_h}, we have
\begin{equation*}
\begin{aligned}
|u_{kh}(T,x_0)| &\le C \lh \|u\|_{L^\infty(I \times \Om)} \Biggl(\sum_{m=1}^{M}\|\pa_t g_{kh}\|_{L^1(I_m; L^1(\Om))}+\|\Delta_h g_{kh}\|_{L^{1}(I; L^1(\Om))}\\
&\qquad\qquad\qquad\qquad\qquad\qquad\qquad\qquad\qquad+\sum_{m=1}^{M-1}\|[g_{kh}]_m\|_{L^1(\Om)} + \|g_{kh,M}^{-}\|_{L^1(\Om)}\Biggr)\\
&\le  C\lh\lk\|u\|_{L^\infty(I \times \Om)} \norm{P_h \tilde \delta}_{L^1(\Om)}\\
& \le  C\lh\lk\|u\|_{L^\infty(I \times \Om)},
\end{aligned}
\end{equation*}
where in the last step we used the stability of the $L^2$ projection $P_h$ with respect to the $L^1(\Omega)$ norm,  see, e.\,g., \cite{DouglasDupontWahlbin:1975} and the fact that $\norm{\tilde \delta}_{L^1(\Omega)}\le C$.

Using  that the cG($r$)dG($q$) method is invariant on $\Xkh$, by replacing $u$ and $u_{kh}$ with $u-\chi$ and $u_{kh}-\chi$ for any $\chi\in \Xkh$, and
using the triangle inequality we obtain
\[
\abs{u(T,x_0)-u_{kh}(T,x_0)} \le C \lk \lh \inf_{\chi \in \Xkh} \norm{u-\chi}_{L^\infty(I\times \Om)}.
\]

\textbf{Case 2, \boldmath{$t_{M-1} < \tilde t <T$}:}

In this case we consider the following regularized Green's function
\begin{equation}\label{eq: heat with dirac on the RHS}
\begin{aligned}
-\pa_t \tilde{g}(t,x)-\Delta \tilde{g}(t,x) &= \tilde{\delta}_{x_0}(x)\tilde{\theta}(t)& (t,x)\in \IOm,\;  \\
\tilde{g}(t,x)&=0, &(t,x) \in I\times\pa\Omega, \\
\tilde{g}(T,x)&=0, & x \in \Omega,
\end{aligned}
\end{equation}
where $\tilde{\theta}\in C^1(\bar I)$ is the regularized Delta function in time with properties
\[
\supp \tilde{\theta} \subset (t_{M-1},T), \quad \|\tilde{\theta}\|_{L^1(I_M)}\le C
\]
and
$$
(\tilde{\theta},\varphi_k)_{I_M}=\varphi_k(\tilde{t}),\quad \forall \varphi_k\in \Ppol{q}(I_M).
$$
Let $\tilde{g}_{kh}$ be cG($r$)dG($q$) approximation of $\tilde{g}$, i.e.
\[
B(\varphi_{kh},\tilde g - \tilde g_{kh}) = 0 \quad \forall \varphi_{kh} \in \Xkh.
\]
Then, using that cG($r$)dG($q$) method is consistent, we have
$$
\begin{aligned}
u_{kh}(\tilde{t},x_0)&=(u_{kh}, \tilde{\delta}_{x_0}\tilde{\theta})=B(u_{kh},\tilde{g})=B(u_{kh},\tilde{g}_{kh})=B(u,\tilde{g}_{kh})\\
&=-\sum_{m=1}^{M}(u,\pa_t\tilde{g}_{kh})_{\ImOm}+(\na  u, \na \tilde{g}_{kh})_{\IOm}-\sum_{m=1}^{M}(u_m,[\tilde{g}_{kh}]_m)_{\Om},
\end{aligned}
$$
where  in the sum with jumps we included the last term by setting $\tilde{g}_{kh,M+1} = 0$ and defining
consequently $[\tilde{g}_{kh}]_M = -\tilde{g}_{kh,M}$.
Similarly to the estimates of $J_1$, $J_2$, $J_3$ above,
using the stability of the Ritz projection in $L^\infty$ norm on polyhedral domains, see Lemma~\ref{lemma: stablity of Ritz}, we have
$$
\begin{aligned}
u_{kh}(\tilde{t},x_0)&=-\sum_{m=1}^{M}(u,\pa_t\tilde{g}_{kh})_{\IOm}+(\na  u, \na \tilde{g}_{kh})_{\IOm}-\sum_{m=1}^{M}(u_m,[\tilde{g}_{kh}]_m)_{\Om}\\
&\le C\lh\|u\|_{L^\infty(\IOm)}\left(\sum_{m=1}^{M}\|\pa_t\tilde{g}_{kh}\|_{L^1(I_m; L^1(\Om))}+\|\Delta_h\tilde{g}_{kh}\|_{L^{1}(I; L^1(\Om))}+
\sum_{m=1}^{M}\|[\tilde{g}_{kh}]_m\|_{L^1(\Om)}\right).
\end{aligned}
$$
Using the discrete maximal parabolic regularity result from Lemma \ref{lemma: fully discrete_maximal_parabolic} with $\vertiii{\cdot}=\|\cdot\|_{L^1(\Om)}$ and $M_h \le C$, cf. Remark~\ref{remark:M_h}, we obtain
\[
u_{kh}(\tilde{t},x_0) \le C\lk\lh\|u\|_{L^\infty(\IOm)}\|P_h\tilde{\delta}_{x_0}\|_{L^1(\Om)}\|\tilde{\theta}\|_{L^1(I_M)}\le C\lk\lh\|u\|_{L^\infty(\IOm)}.
\]
As in the first case this implies
\[
\abs{u(\tilde t,x_0)-u_{kh}(\tilde t,x_0)} \le C \lk \lh \inf_{\chi \in \Xkh} \norm{u-\chi}_{L^\infty(I\times \Om)}.
\]
This completes the proof of the theorem.
\section{Proof of Theorem \ref{thm:local_best_approx}}\label{sec: proofs local results}

To obtain the interior estimate we introduce a smooth cut-off function $\omega$ with the properties that
\begin{subequations} \label{def: properties of omega a local}
\begin{align}
\omega(x)&\equiv 1,\quad x\in B_d \label{eq: property 1 of omega_a local}\\
\omega(x)&\equiv 0,\quad x\in \Om\setminus B_{2d} \label{eq: property 2 of omega_a local}\\
 |\na \omega|&\le Cd^{-1}, \quad |\na^2 \omega|\le Cd^{-2}, \label{eq: property 3 of omega_a local}
\end{align}
\end{subequations}
where $B_d=B_d(x_0)$ is a ball of radius $d$ centered at $x_0$.

As in the proof of Theorem~\ref{thm:global_best_approx} we consider two cases: $\tilde t = T$ and $t_{M-1} < \tilde t < T$. In the first case we obtain
\begin{equation}\label{eq: local starting expression}
u_{kh}(T,x_0)=B(u_{kh}, g_{kh}) =B(u, g_{kh})=B(\om u, g_{kh})+B((1-\om)u, g_{kh}),
\end{equation}
where $g$ is the solution of~\eqref{eq: heat with dirac initial}  and $g_{kh}\in \Xkh$ is the solution of~\eqref{eq: discrete heat with dirac initial}. The first term can be estimated using the global result from Theorem~\ref{thm:global_best_approx}. To this end we introduce $\tilde u=\omega u$ and the cG($r$)dG($q$) solution $\tilde u_{kh} \in \Xkh$ defined by
\[
B(\tilde u_{kh} - \tilde u, \varphi_{kh}) = 0 \quad \text{for all }  \varphi_{kh} \in \Xkh.
\]
There holds
\[
B(\tilde u, g_{kh}) = B(\tilde u_{kh}, g_{kh}) = \tilde u_{kh}(T,x_0)
\le   C \lk \lh  \norm{\tilde{u}}_{L^\infty(I\times \Om)}\le C \lk \lh  \norm{u}_{L^\infty(I\times B_{2d})}.
\]
This results in
\begin{equation}\label{eq:local_after_first_step}
|u_{kh}(T,x_0)| \le C \lk \lh  \norm{u}_{L^\infty(I\times B_{2d})} + B((1-\om)u, g_{kh}).
\end{equation}
It remains to estimate the term $B((1-\om)u, g_{kh})$. Using the dual expression~\eqref{eq:B_dual} of the bilinear form $B$ we obtain
\begin{equation}\label{eq:J1_J2_local}
\begin{aligned}
B((1-\om)u, g_{kh}) &=-\sum_{m=1}^{M}((1-\om)u,\pa_t g_{kh})_{\ImOm}+(\na((1-\om)u), \na g_{kh})_{\IOm}
\\&- \sum_{m=1}^{M-1}((1-\om)u_m, [g_{kh}]_m)_{\Om} +((1-\om)u(T), g_{kh,M}^-)_{\Om}\\
&=-\sum_{m=1}^{M}(\si^{-\frac{N}{2}}(1-\om)u,\si^{\frac{N}{2}}\pa_t g_{kh})_{\ImOm}+(\na((1-\om)u), \na g_{kh})_{\IOm}
\\&- \sum_{m=1}^{M-1}(\si^{-\frac{N}{2}}(1-\om)u_m, \si^{\frac{N}{2}}[g_{kh}]_m)_{\Om} +(\si^{-\frac{N}{2}}(1-\om)u(T), \si^{\frac{N}{2}}g_{kh,M}^-)_{\Om}\\
&= J_1 + J_2+J_3+J_4.
\end{aligned}
\end{equation}
For $J_1$, using that $\si^{-\frac{N}{2}}\le Cd^{-\frac{N}{2}}$ on $\supp (1-\omega) \subset \Om\setminus B_d$ and $(1-\om)\le 1$,
we obtain
\begin{equation}\label{eq:J1_local}
\begin{aligned}
J_1 & \le \norm{\si^{-\frac{N}{2}}(1-\om)u}_{L^\infty(I;L^2(\Om))}\sum_{m=1}^M\|\si^{\frac{N}{2}}\pa_t g_{kh}\|_{L^1(I_m; L^2(\Om))}\\
&\le Cd^{-\frac{N}{2}}\|u\|_{L^\infty(I;L^2(\Om))}\sum_{m=1}^M\|\si^{\frac{N}{2}}\pa_t g_{kh}\|_{L^1(I_m; L^2(\Om))}.
\end{aligned}
\end{equation}
To estimate $J_2$, we define $\psi = (1-\om)u$ and proceed using the Ritz projection $R_h$ defined by~\eqref{eq:Ritz_proj}. There holds
\[
\begin{aligned}
(\na\psi(t), \na g_{kh}(t))_{\Om}&=(\na R_h\psi(t), \na g_{kh}(t))_{\Om}=-(R_h\psi(t), \De_h g_{kh}(t))_{\Om}\\
&=-(R_h\psi(t), \De_h g_{kh}(t))_{B_{d/2}}-(R_h\psi(t), \De_h g_{kh}(t))_{\Om\setminus B_{d/2}}\\
&\le \|R_h\psi(t)\|_{L^\infty(B_{d/2})} \|\De_h g_{kh}(t)\|_{L^1(B_{d/2})}\\
&\qquad\qquad\qquad\qquad+Cd^{-\frac{N}{2}}\|R_h\psi(t)\|_{L^2(\Om\setminus B_{d/2})}\|\si^{\frac{N}{2}}\De_h g_{kh}(t)\|_{L^2(\Om\setminus B_{d/2})}\\
&\le \|R_h\psi(t)\|_{L^\infty(B_{d/2})} \|\De_h g_{kh}(t)\|_{L^1(\Omega)}+Cd^{-\frac{N}{2}}\|R_h\psi(t)\|_{L^2(\Om)}\|\si^{\frac{N}{2}}\De_h g_{kh}(t)\|_{L^2(\Om)},
\end{aligned}
\]
where we used $\si^{-\frac{N}{2}}\le Cd^{-\frac{N}{2}}$ on $\Om\setminus B_{d/2}$.
In the interior pointwise error estimates ~\cite[Thm. 1.1]{AHSchatz_LBWahlbin_1995a} with $F\equiv 0$, choosing $\chi=0$, $s=0$, $q=2$ and using the triangle inequality and the fact that $\supp \psi(t) \subset \Omega \setminus B_d$, we have
$$
\|R_h\psi(t)\|_{L^\infty(B_{d/2})} \le C\lh\|\psi(t)\|_{L^\infty(B_{d})}+Cd^{-\frac{N}{2}}\|R_h\psi(t)\|_{L^2(\Omega)}=Cd^{-\frac{N}{2}}\|R_h\psi(t)\|_{L^2(\Omega)}.
$$
Using a standard elliptic estimate and recalling $\psi = (1-\om) u$ we have
\[
\begin{aligned}
\|R_h \psi(t)\|_{L^2(\Om)} &\le \|\psi(t)\|_{L^2(\Om)} + \|\psi(t)-R_h\psi(t)\|_{L^2(\Om)}\\
&\le  \|\psi(t)\|_{L^2(\Om)} + c h \|\nabla \psi(t)\|_{L^2(\Om)}\\
&\le \|u(t)\|_{L^2(\Om)} + c h \|(1-\om)\nabla u(t) - \nabla \om u(t)\|_{L^2(\Om)}\\
&\le c \|u(t)\|_{L^2(\Om)} + ch \|\nabla u(t)\|_{L^2(\Om)},
\end{aligned}
\]
where in the last step we used $\abs{\nabla \om} \le C d^{-1} \le C h^{-1}$.

Therefore we obtain
\[
(\na\psi(t), \na g_{kh}(t))_{\Om} \le C d^{-\frac{N}{2}} \left( \|u(t)\|_{L^2(\Om)} + ch \|\nabla u(t)\|_{L^2(\Om)} \right) \left( \|\De_h g_{kh}(t)\|_{L^1(\Omega)} + \|\si^{\frac{N}{2}}\De_h g_{kh}(t)\|_{L^2(\Om)}\right).
\]
This results in
\begin{equation}
J_2 \le C d^{-\frac{N}{2}} \left( \|u\|_{L^\infty(I;L^2(\Om))} + ch \|\nabla u\|_{L^\infty(I;L^2(\Om))} \right) \left( \|\De_h g_{kh}\|_{L^1(I;L^1(\Omega))} + \|\si^{\frac{N}{2}}\De_h g_{kh}\|_{L^1(I;L^2(\Om))}\right).
\end{equation}
For $J_3$, similarly to $J_1$ we obtain
\begin{equation}\label{eq:J3_local}
\begin{aligned}
J_3&\le \|\si^{-\frac{N}{2}}(1-\omega)u\|_{L^\infty(I;L^2(\Om))}\sum_{m=1}^{M-1}\|\si^{\frac{N}{2}} [g_{kh}]_m\|_{L^2(\Om)}\\
&\le Cd^{-\frac{N}{2}}\|u\|_{L^\infty(I;L^2(\Om))}\sum_{m=1}^{M-1}\|\si^{\frac{N}{2}} [g_{kh}]_m\|_{L^2(\Om)}.
\end{aligned}
\end{equation}
Finally,
\begin{equation}
J_4\le Cd^{-\frac{N}{2}}\|u\|_{L^\infty(I;L^2(\Om))}\|\si^{\frac{N}{2}}g_{kh,M}^-\|_{L^2(\Om)}.
\end{equation}
Combining the estimates for $J_1$, $J_2$, $J_3$, and $J_4$,  we have
\begin{multline*}
B((1-\om)v, g_{kh})\le C d^{-\frac{N}{2}} \left( \|u\|_{L^\infty(I;L^2(\Om))} + ch \|\nabla u\|_{L^\infty(I;L^2(\Om))} \right) \\
\times \Biggl(\sum_{m=1}^M\|\si^{\frac{N}{2}}\pa_t g_{kh}\|_{L^1(I_m; L^2(\Om))}+ \|\De_h g_{kh}\|_{L^1(I;L^1(\Omega))} + \|\si^{\frac{N}{2}}\De_h g_{kh}\|_{L^1(I;L^2(\Om))}\\ + \sum_{m=1}^{M-1}\|\si^{\frac{N}{2}} [g_{kh}]_m\|_{L^2(\Om)} + \|\si^{\frac{N}{2}}g_{kh,M}^-\|_{L^2(\Om)}\Biggr).
\end{multline*}
For the term $\|\De_h g_{kh}\|_{L^1(I;L^1(\Omega))}$ we apply Lemma~\ref{lemma: smoothing L1 in time}  with $\vertiii{\cdot}=\|\cdot\|_{L^1(\Om)}$ and $M_h \le C$ and for all weighted terms with $\vertiii{\cdot}=\|\si^{\frac{N}{2}}(\cdot)\|_{L^2(\Om)}$ and $M_h \le C \lh$, cf. Remark~\ref{remark:M_h}, resulting in
\[
\begin{aligned}
B((1-\om)v, g_{kh}) &\le C d^{-\frac{N}{2}} \lk \lh \left(\norm{u}_{L^\infty(I;L^2(\Om))} + h \norm{\nabla u}_{L^\infty(I;L^2(\Om))} \right)\left(\norm{P_h \tilde \delta}_{L^1(\Omega)} + \norm{\sigma^{\frac{N}{2}} P_h \tilde \delta}_{L^2(\Omega)}\right)\\
& \le  C d^{-\frac{N}{2}} \lk \lh \left(\norm{u}_{L^\infty(I;L^2(\Om))} + h \norm{\nabla u}_{L^\infty(I;L^2(\Om))} \right),
\end{aligned}
\]
where in the last step we again used the stability of the $L^2$ projection with respect to the $L^1$ norm, the fact that $\norm{\tilde \delta}_{L^1(\Omega)}\le C$, and Lemma~\ref{lemma:sigma_delta} for the term $\norm{\sigma^{\frac{N}{2}} P_h \tilde \delta}_{L^2(\Omega)}$.
Inserting this inequality into~\eqref{eq:local_after_first_step}, we obtain
$$
|u_{kh}(T,x_0)| \le C \lk \lh \left(
\|u\|_{L^\infty(I\times B_{2d})}+ d^{-\frac{N}{2}} \left(\norm{u}_{L^\infty(I;L^2(\Om))} + h \norm{\nabla u}_{L^\infty(I;L^2(\Om))} \right)\right).
$$
Using that the cG($r$)dG($q$) method is invariant on $\Xkh$, by replacing $u$ and $u_{kh}$ with $u-\chi$ and $u_{kh}-\chi$ for any $\chi\in \Xkh$, we obtain Theorem~\ref{thm:local_best_approx} for the case $\tilde t = T$.

In the case $t_{M-1} < \tilde t < T$ we proceed as in the proof of Theorem~\ref{thm:global_best_approx} using the dual problem~\eqref{eq: heat with dirac on the RHS} instead of~\eqref{eq: heat with dirac initial}. Then, we proceed as in the above proof using in the last step the discrete maximal parabolic regularity from Lemma~\ref{lemma: fully discrete_maximal_parabolic} instead of Lemma~\ref{lemma: smoothing L1 in time}. This completes the proof.


\bibliography{lit_Linf}
\bibliographystyle{siam}

\end{document}